 \documentclass[11pt]{article}

\usepackage[margin=1.3in]{geometry}
\usepackage{amsmath, amsfonts, amssymb, amsthm}
\usepackage{graphicx}
\usepackage{hyperref}
\usepackage{float}
\hypersetup{
    colorlinks,%
    citecolor=black,%
    filecolor=black,%
    linkcolor=black,%
    urlcolor=black
}
\usepackage{color}
\usepackage{mathrsfs}
\usepackage{textcomp}
\usepackage{color}
\usepackage{marginnote}

\usepackage{verbatim}

\usepackage{sectsty}

\makeatletter

\newdimen\bibspace
\renewenvironment{thebibliography}[1]{%
 \section*{\refname 
       \@mkboth{\MakeUppercase\refname}{\MakeUppercase\refname}}%
     \list{\@biblabel{\@arabic\c@enumiv}}%
          {\settowidth\labelwidth{\@biblabel{#1}}%
           \leftmargin\labelwidth
           \advance\leftmargin\labelsep
           \itemsep\bibspace
           \parsep\z@skip     %
           \@openbib@code
           \usecounter{enumiv}%
           \let\p@enumiv\@empty
           \renewcommand\theenumiv{\@arabic\c@enumiv}}%
     \sloppy\clubpenalty4000\widowpenalty4000%
     \sfcode`\.\@m}
    {\def\@noitemerr
      {\@latex@warning{Empty `thebibliography' environment}}%
     \endlist}

\makeatother

\makeatletter

\newtheorem{thm}{Theorem}[section]
\newtheorem{lem}[thm]{Lemma}
\newtheorem{prop}[thm]{Proposition}
\newtheorem{defn}[thm]{Definition}

\newtheorem{rem}[thm]{Remark}

\newtheorem*{rem*}{Remark}
\numberwithin{equation}{section}

\def\XXint#1#2#3{{\setbox0=\hbox{$#1{#2#3}{\int}$}
  \vcenter{\hbox{$#2#3$}}\kern-.5\wd0}}

           \newcommand{\ud}{\mathrm{d}}
\newcommand{\be}{\begin{equation}}      \newcommand{\ee}{\end{equation}}

\begin{document}

\title{\textbf{On a Class of Nonlocal SIR Models}
\bigskip}

\author{Li Guan,\quad Dong Li \footnote{DL is partially supported by Hong Kong RGC grant GRF 16307317.}, \quad  Ke Wang\footnote{KW is partially supported by HKUST Initiation Grant IGN16SC05.}, \quad
Kun Zhao\footnote{KZ is partially supported by Simons Foundation Collaboration Grant for Mathematicians 413028.}}

\date{\today}

\maketitle

\begin{abstract}
We revisit the classic Susceptible-Infected-Recovered (SIR) epidemic model and
one of its nonlocal variations recently developed in \cite{Guan}. We introduce several
new approaches to derive exact analytical solutions in the classical situation and 
analyze the corresponding effective approximations in the nonlocal setting. 
An interesting new feature of the nonlocal models, compared with the classic SIR model,
is the appearance of multiple peak solutions for the infected population. We provide
several rigorous results on the existence and non-existence of peak solutions with sharp
asymptotics.



\end{abstract}


\section{Introduction}

This paper is oriented around the classic SIR model in mathematical epidemiology and one of its nonlocal variations recently developed in \cite{Guan}. The main purpose is to provide a systematic approach for finding exact analytical solutions to the models under appropriate initial conditions. 
The susceptible-infected-recovered (SIR) model has been a cornerstone in the study of the spreading mechanisms of infectious diseases for many decades since its initiation in the 1920s \cite{KK27}. The success of the SIR model has been a consequence of its intuitive simplicity, analytical tractability, and ability to predict the underlying mechanisms elucidating the spread of infectious diseases. In its simplest form, the classic SIR model is a system of nonlinear ordinary differential equations:
\begin{align}\label{e0}
\dfrac{dS}{dt}&=-\lambda \frac{I}{N}S, \notag\\
\dfrac{dI}{dt}&=\lambda \frac{I}{N}S-\gamma I, \\
\dfrac{dR}{dt}&=\gamma I, \notag
\end{align}
where $S$, $I$, $R$ and $N=S+R+I$ denote the susceptible, infected, recovered and total populations, respectively; $\lambda>0$ is the total transmission rate, and $\gamma>0$ is the average recovery rate. One
should note that due to the conservation law $S+I+R\equiv N$ the degree of freedom in \eqref{e1}
is in fact two instead of three. By a normalization of $\frac{S}{N}\to S$, $\frac{I}{N}\to I$, $\frac{R}{N}\to R$, the system \eqref{e0} becomes
\begin{align}\label{e1}
\dfrac{dS}{dt} &=-\lambda IS, \notag\\
\dfrac{dI}{dt} &=\lambda IS-\gamma I,\\
\dfrac{dR}{dt} &=\gamma I, \notag
\end{align}
and the  conservation law becomes $S+I+R\equiv 1$. Despite its simple looking structure, the nonlinear interaction between the susceptible and infected populations makes the investigation of some of the fundamental mathematical properties of the model a daunting task. In particular, the search of closed form (explicit) solutions to the model subject to appropriate initial conditions has been a significant challenge in the community of mathematical epidemiology, yet the resolution is still elusive. To the authors' knowledge, the implicit solution discovered by Harko {\it et al} in 2014 \cite{HLM} is the first exact analytical solution to the model since the 1920s. Although explicit solutions are still not available, the implicit solution can be charted in a systematic way such that it can be applied in various situations for practical purposes.  An exact solution of a special case of the SIR model was presented earlier in \cite{SKS}.  The SIR model has also been investigated numerically in a number of works such as \cite{BB, JB, RDG, RDG2}. A stochastic epidemic-type model was recently studied in \cite{WMM}.  Lyapunov functions for classic SIR and SIS epidemiological models were introduced
in \cite{KW}, and a new one later in \cite{OKK}. For reaction-diffusion type models
one can see recent \cite{Wang18, Lou1, Lou2} and the references therein for more detailed
discussions.


\subsection{Motivation and Goal}

Now we would like to point out the facts that motivate the current work. Recently in \cite{Guan}, the authors proposed an alternative version of the classic SIR model by replacing the exponential function in the integral form of the infected population by rewriting the constant recovery rate as the hazard function of the exponential function, due to their equivalence, and then replacing the hazard function  of exponential function by other types of hazard functions. The hybrid differential-integral model is more realistic than the classic SIR model in the sense that the former one replaces the constant recovery rate in the latter one by a probability density function. Indeed, when the recovery rate is constant, by replacing the nonlinear term in the second equation of \eqref{e1} by $-S'(t)$, then formally solving the equation for $I(t)$, one gets
$$
I(t)=\int_0^t e^{-\gamma(t-\tau)}[-S'(\tau)]d\tau+I_0e^{-\gamma t},
$$
where $'=\frac{d}{dt}$. The main idea of \cite{Guan} is to rewrite the exponential function as $e^{-\gamma t}=1-(1-e^{-\gamma t})$, then replace $1-e^{-\gamma t}$, which is exactly the cumulative density function of exponential distribution, by a  cumulative density function supported in $[0,\infty)$, called $G(t)$, together with the differential equation for $S$ and the conservation of total population resulting in the nonlocal (normalized) SIR model:
\begin{equation}\label{e2}
\begin{aligned}
\frac{dS}{dt}&=-\lambda IS,\\
I(t)&=\int_0^t [1-G(t-\tau)]\cdot[-S'(\tau)]d\tau+I_0[1-G(t)],\\
R&=1-I-S.
\end{aligned}
\end{equation}
We observe that once the constant recovery rate is replaced by a probability density function, the approach of \cite{HLM} (which is built on analyzing the usual Abel-type equations) becomes unaccessible for finding analytical solutions to the nonlocal model. 
Furthermore, even the construction and analysis of the corresponding approximate solutions 
seem nontrivial without deeper understanding of the nonlocal model.
These are the major facts that motivate the current work. 

The goal of this paper is to provide two alternative (more systematic) approaches for finding analytic and approximate solutions to the classic and nonlocal SIR models, \eqref{e1} $\&$ \eqref{e2}, under appropriate initial conditions. The first approach (see Sections 2 and 4), is to  regard the infected population, $I$, as a function of the susceptible one, $S$, instead of $t$, based on the observation that $S$ is a strictly decreasing function of $t$ (hence the inverse function of $S(t)$ exists). Then by solving a differential equation for $I$ in terms of $S$, one gets a closed form of $I$ in terms of $S$. 
In the nonlocal setting, one has to make certain approximations of the nonlocal kernel in order to
derive a quantitative explicit expression. In the second approach (see Section 3 and esp. Section 5),
we introduce a rescaled (nonlocal) time variable $\tau$ which is measured according to normalized instantaneous
infected population, and analyze directly the susceptible population (along with other variables) by
an approximation of the interaction kernel. All properties of the solutions are analyzed in
the $\tau$-scale. We then translate these results into the original time scale $t$ by using the nonlocal
map $t\to \tau$.  Conceptually, this approach is rather clean since it corresponds to a nonlocal stretching of the time axis under which the dynamics of the model becomes much simpler (this is
particularly transparent in the classic SIR case, see Section 3). 
On the other hand, an interesting new feature of the nonlocal model is that one
can accommodate certain infected populations with several peaks (in the classic SIR case, one can have
at most one, see Section 3). We give several rigorous constructions on the existence of such solutions
having multiple peaks (with ``controlled" centers)  in Section 5. 


The rest of this paper is organized as follows. In Sections 2 and 3, we revisit the classic SIR model and introduce the aforementioned analyses in the classic situation. These analyses are then generalized
to the nonlocal SIR models in Sections 4 and 5. Some concluding remarks are gathered at the
end of Section 5.

\subsection*{Notation.} For any two positive quantities $X$ and $Y$, we shall use the notation
$X\lesssim Y$ if $X\le CY$ for some harmless positive constant $C$.  We shall write $X\ll Y$ 
if $X \le c Y$ for some sufficiently small constant $c>0$. The smallness of the constant $c$
is usually clear from the context.


\section{Classic SIR Model}\label{sec:SIR}

In this section, we revisit the classic SIR model which will motivate
  the nonlocal model in later sections. The classic (normalized) SIR model takes the form:
\begin{equation}\label{sir}
\begin{aligned}
\frac{dS}{dt}&=-\lambda IS,\\
\frac{dI}{dt}&=\lambda IS-\gamma I,\\
\frac{dR}{dt}&=\gamma I.
\end{aligned}
\end{equation}
We observe that since the transmission rate $\lambda$ is positive, in the biologically relevant regime where the populations are positive, the susceptible population is strictly decreasing with respect to $t$ due to the first equation in \eqref{sir}. This indicates that $S(t)$ is a one-to-one function, and hence one can define an inverse function of $S$ and express $t$ in terms of $S$. Therefore, we can regard $I$ as a function of $S$, instead of $t$. By dividing the second equation in \eqref{sir} by the first one, we obtain
\begin{equation}\label{sir1}
\begin{aligned}
\frac{dI}{dS}=\dfrac{\lambda IS-\gamma I}{-\lambda IS},
\end{aligned}
\end{equation}
from which we get
\begin{equation}\label{sir2}
\begin{aligned}
\frac{dI}{dS}=-1+\dfrac{\gamma}{\lambda S}.
\end{aligned}
\end{equation}
By solving \eqref{sir2} for $I$, we have
\begin{equation}\label{sir3}
\begin{aligned}
I(S)=-S+\dfrac{\gamma}{\lambda}\ln S+C,
\end{aligned}
\end{equation}
where 
$$
C=I_0+S_0-\frac{\gamma}{\lambda}\ln S_0.
$$
Now, by plugging \eqref{sir3} into the first equation in \eqref{sir}, we have
\begin{equation}\label{sir4}
\begin{aligned}
\frac{dS}{dt}=\lambda\left(S^2-\dfrac{\gamma }{\lambda}S\ln S-CS\right).
\end{aligned}
\end{equation}
Denote the anti-derivative of 
$$
\dfrac{1}{\lambda S^2-\gamma  S\ln S-\lambda CS}
$$
by $F(S)$. Then we deduce from \eqref{sir4} that 
\begin{equation}\label{sir5}
F(S)-F(S_0)=t,
\end{equation}
which is consistent with the analytical solution reported in \cite{HLM}. Once the susceptible population is computed from \eqref{sir5}, the infected population is then calculated from \eqref{sir3}, and $R$ is determined from the conservation of total population: $R=1-I-S$.


\section{New Analysis for Classic SIR Model}\label{sec:newSIR}
Consider the equations of $S$ and $I$ in \eqref{sir}
\begin{align*}
\begin{cases}
\frac d {dt}  S= -\lambda  I  S\\
\frac d {dt}  I = \lambda  I  S - \gamma  I,\\
 S|_{t=0} = S_0, \;  I|_{t=0} = I_0.
\end{cases}
\end{align*}
Note that $ S_0+  I_0 \le 1$ and $ S_0\ge 0$, $ I_0\ge 0$. Without loss of generality, we 
may assume $ I_0 >0$ since if $ I_0=0$ then $( S_0, 0)$ is a steady
state. Also we may assume $ S_0>0$ since the case $ S_0=0$ generates only
the trivial dynamics $( S(t),  I(t) ) = (0,  I_0 e^{-\gamma t})$. Now observe that
since $ I_0>0$, we have 
\begin{align*}
0< I(t)<1, \qquad \forall\, 0\le t<\infty.
\end{align*}
We now intend to \emph{linearise} the nonlinear dynamics by introducing a new (nonlinear
and nonlocal) time scale $\tau$ as
\begin{align*}
\tau (t) = \int_0^t  I(s) ds, \qquad\text{i.e.} \quad d\tau =  I dt.
\end{align*}
We then introduce functions $(S_{1}, I_1)$ such that  $ S(t) = S_1  (\tau(t))$, $ I(t) =I_1 (\tau(t) )$. Then
\begin{align*}
\begin{cases}
\frac d {d\tau} S_1 = -\lambda S_1, \\
\frac d {d\tau} I_1 = \lambda S_1-\gamma, \\
S_1 |_{\tau=0} =  S_0, \; \;I_1 |_{\tau=0} =  I_0.
\end{cases}
\end{align*}
Solving the above system yields 
\begin{align} \label{S1I1}
\begin{cases}
S_1(\tau) =  S_0 e^{-\lambda \tau}, \\
 I_1 (\tau) = S_0 +  I_0 - S_0 e^{-\lambda \tau} - \gamma \tau.
 \end{cases}
\end{align}
Observe that $I_1^{\prime\prime} <0$ and $I_1^{\prime}(0) = \lambda  S_0-\gamma$.
Now discuss two cases.

Case 1: $\lambda  S_0 > \gamma$.  In the literature this corresponds to the case where the
reproduction number is bigger than $1$. Clearly $I_1 $ will first increase and then decrease 
to zero.  
\begin{figure}[H]
  \centering
   \includegraphics[width=6cm]{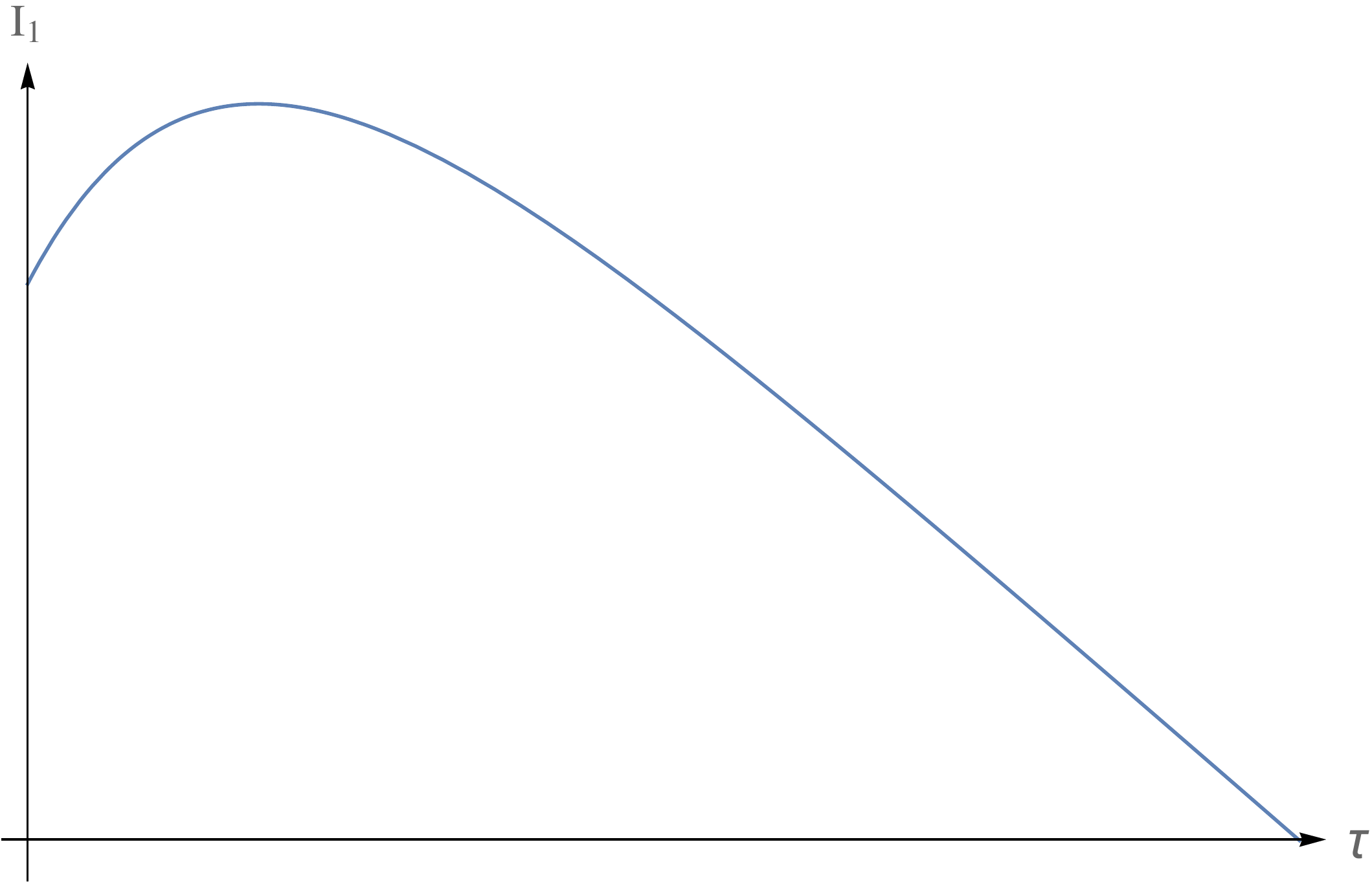}
     \caption{The function $I_1 $ in Case 1}\label{fig:1}
\end{figure}

Case 2: $\lambda  S_0\le \gamma$. This corresponds to the case where the reproduction
number is less than or equal to $1$. In this case $I_1 $ will monotonically decrease to zero.
\begin{figure}[H]
  \centering
   \includegraphics[width=6cm]{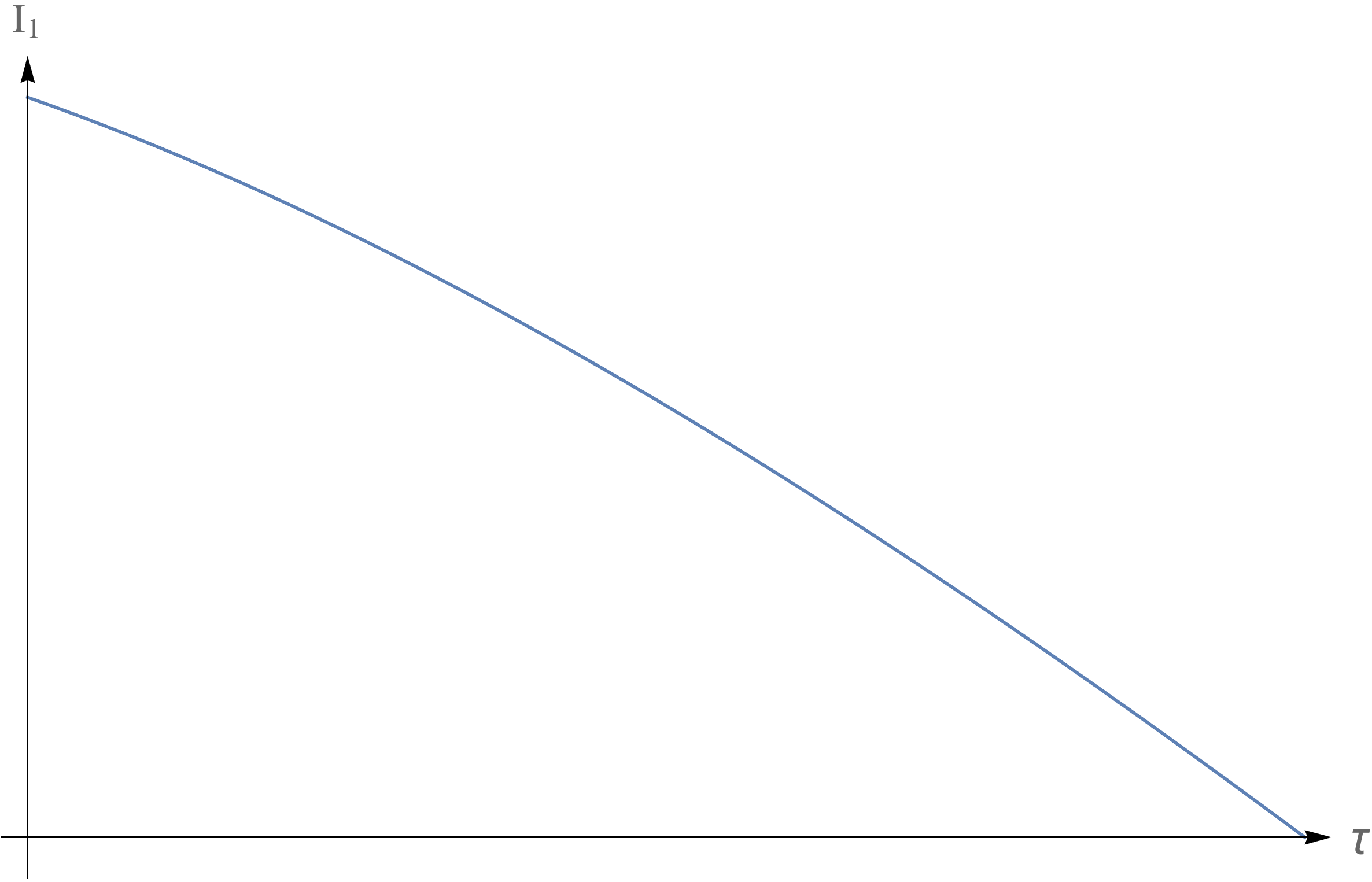}
     \caption{The function $I_1 $ in Case 2}\label{fig:2}
\end{figure}

Concluding from the above two cases, it is clear that $I_1 (\tau)$ is defined in the regime
$0\le \tau \le \tau_{\infty}$, where $\tau_{\infty}$ is the unique number solving the 
equation
\begin{align*}
 S_0 + I_0 -  S_0 e^{-\lambda \tau_{\infty}} - \gamma \tau_{\infty} =0.
\end{align*}

To recover $ S(t) =S_1(\tau(t) )$, one may first discretize the interval $[0,\tau_{\infty}]$
as $\tau_0, \tau_1,\cdots, \tau_m$, and solve
\begin{align*}
t_j = \int_0^{\tau_j} \frac 1 {{I_1} (s) } ds, \quad j=0,\cdots, m.
\end{align*}
Since $ S(t_j) = S_1(\tau_j)$ and $ I (t_j) = I_1 (\tau_j)$, the dynamics of $( S,
 I)$ is then fully recovered. We should emphasize that through the rescaled
time ``clock" $\tau$, the dynamics for the original SIR model is fully captured by
\eqref{S1I1}, and the map $\tau \to t$ simply provides the abscissa of the original SIR variables.

\begin{rem}
Similar analysis can be conducted on the classic SIS (cf. \cite{H76}) epidemic model:
\begin{align*}
\begin{cases}
\frac {dS}{dt} = -\frac {\lambda}  N S I + (\mu+\gamma)I,\\
\frac {dI} {dt} = \frac {\lambda}{N} SI - (\mu+ \gamma)I,
\end{cases}
\end{align*}
where $\lambda>0$, $\gamma>0$, $\mu\ge 0$ correspond to total transmission rate, average recovery rate and average death rate, respectively. Due to the conservation of $S(t)+I(t) \equiv N$, it suffices to study
only the dynamics of $S$. In terms of the rescaled variable $\tilde S= S/N$ and the rescaled
``clock" $d\tau =\tilde I dt$ ($\tilde I=I/N$), one has
\begin{align*}
\frac {d \tilde S} {d\tau} = - \lambda  \tilde S + \mu+\gamma
\end{align*}
which yields
\begin{align*}
&\tilde S (\tau) = \frac {\mu+\gamma}{\lambda} (1-e^{-\lambda \tau} ) + \tilde S_0 e^{-\lambda \tau},\\
&\tilde I (\tau)= 1-\frac{\mu+\gamma}{\lambda} +e^{-\lambda \tau}
 (\frac {\mu+\gamma}{\lambda} - \tilde S_0).
\end{align*}
One can then determine $\tau$ by using the relation $d\tau =\tilde I dt$. We omit the 
details. We should remark that alternatively one can regard $\tilde S$ 
as a function of $\tilde I$ and solve the ODE directly. 

\end{rem}

\section{Nonlocal SIR Model}\label{sec:hsir}

In this section, we use the idea in Section \ref{sec:SIR} to derive analytical solutions to the 
(normalized) nonlocal differential-integral SIR model:
\begin{equation}\label{hsir}
\begin{aligned}
\frac{dS}{dt}&=-\lambda IS,\\
I(t)&=\int_0^t [1-G(t-\tau)]\cdot[-S'(\tau)]d\tau+I_0[1-G(t)],\\
R&=1-I-S,
\end{aligned}
\end{equation}
which becomes the classic normalized SIR model when $G(t)=1-e^{-\gamma t}$. Note that the integral equation for $I$ can be written as
\begin{equation}\label{hsir1}
\begin{aligned}
I(t)&=S_0-S(t)+\int_0^t G(t-\tau)S'(\tau)d\tau+I_0[1-G(t)]\\
&=1-S(t)+\int_0^t G(t-\tau)S'(\tau)d\tau-I_0G(t)\quad\mbox{by recalling } R_0=0.
\end{aligned}
\end{equation}
By regarding $t$ as a function of $S$, we have an alternative expression for $I$:
\begin{equation}\label{hsir2}
\begin{aligned}
I(S)=1-S+\int_{S_0}^{S} \mathbb{\tilde G}(S, \mathbb{S})d\mathbb{S}-I_0\mathbb{G}(S),
\end{aligned}
\end{equation}
which is a nonlocal equation with respect to $S$. Hence, one can postulate various conditions on $\mathbb{\tilde G}$ to generate different kinds of SIR models, depending on the specific biological environments under consideration. To proceed we consider the following \emph{convolution}-type
kernel
\begin{align*}
\mathbb{\tilde G}(S,\mathbb{S}) = \mathbb G (S-\mathbb S),
\end{align*}
and \eqref{hsir2} becomes (after a change of variable)
\begin{align}\label{hsir2a}
\begin{aligned}
I(S)=1-S+\int_{0}^{S-S_0} \mathbb{ G}(\mathbb{S})d\mathbb{S}-I_0\mathbb{G}(S),
\end{aligned}
\end{align}

Differentiating \eqref{hsir2a} with respect to $S$, we obtain
\begin{equation}\label{hsir3}
\begin{aligned}
\frac{dI}{dS}
&=-1+\mathbb{G}(S-S_0)-I_0\mathbb{G}'(S).
\end{aligned}
\end{equation}
In view of \eqref{sir2} we see that when the kernel $\mathbb{G}$ satisfies
\begin{equation}\label{hsir4}
\mathbb{G}(S-S_0)-I_0\mathbb{G}'(S)=\frac{\gamma }{\lambda S},
\end{equation}
the following hybrid model
\begin{equation}\label{hsir5}
\begin{aligned}
\frac{dS}{dt}&=-\lambda IS,\\
I(S)&=1-S+\int_{S_0}^{S} \mathbb{G}(S-\mathbb{S})d\mathbb{S}-I_0\mathbb{G}(S),\\
R&=1-I-S
\end{aligned}
\end{equation}
coincides with the classic normalized SIR model \eqref{sir}. By plugging the second equation of \eqref{hsir5} into the first equation, we have
 \begin{equation}\label{hsir6}
\begin{aligned}
\frac{dS}{dt}=-\lambda S\left(1-S+\int_{S_0}^{S} \mathbb{G}(S-\mathbb{S})d\mathbb{S}-I_0\mathbb{G}(S)\right),
\end{aligned}
\end{equation}
from which one gets
 \begin{equation}\label{hsir7}
\begin{aligned}
\dfrac{1}{\lambda S\left(1-S+\int_{S_0}^{S} \mathbb{G}(S-\mathbb{S})d\mathbb{S}-I_0\mathbb{G}(S)\right)}dS=-dt,
\end{aligned}
\end{equation}
whose solution can be found depending on the specific form of $\mathbb{G}$ that one postulates.


\section{An Approximation of the Nonlocal SIR Model}

In this section, we use a similar approach as in Section \ref{sec:newSIR} to derive new approximating 
solutions to the hybrid differential-integral SIR model in \eqref{hsir}. We shall first derive an
effective ``$\tau$-model" by using a nonlocal stretching of the time axis. The properties of
the $\tau$-model will then be analyzed in detail by some judicious choices of the interaction
kernels.

\subsection{Derivation of the $\tau$-Model}
We introduce a new time scale $\tau$ as
\begin{align*}
\tau (t) = \int_0^t  I(s) \,\ud s, \qquad\text{i.e.} \quad \ud\tau =  I \ud t.
\end{align*}
It is clear that for all practical situations $I(t)$ is positive 
as long as $S_0\ne  0$ and $I_0 \ne 0$ (recall in the second equation
of \eqref{hsir}, one has $I(t)=
\int_0^t (1-G(t-\tilde t)) (\lambda I(\tilde t) S(\tilde t) ) d\tilde t
+I_0 (1-G(t) )$). Hence, $\tau$ is a strictly increasing function, and thus, $\tau$ has an inverse function, which will be denoted as $\varphi$.  Note that $\varphi(0)=0$.
 
Let $ S(t) = S_1 (\tau(t))$, $ I(t) =I_1(\tau(t) )$. Then it follows from \eqref{hsir} that
\begin{align*}
S_1 (\tau)&=S_0 e^{- \lambda\tau };\\
I_1 (\tau)&=\lambda \int_0^\tau[1-G(\varphi(\tau)-\varphi(\tilde\tau))]S_1 (\tilde\tau)\,\ud \tilde\tau+I_0[1-G(\varphi(\tau))].
\end{align*}

Denote $G_1(\tau,\tilde \tau) = G( \varphi(\tau)- \varphi(\tilde \tau) )$.  Then we obtain
\begin{align*}
I_1 (\tau)&=\lambda S_0\int_0^\tau[1- G_1(\tau, \tilde\tau)]e^{- \lambda\tilde\tau}\,\ud \tilde\tau+I_0[1- G_1(\tau,0 )]\\
&=\lambda S_0e^{-\lambda\tau}\int_0^\tau[1- G_1(\tau, \tau-\tilde\tau)]e^{ \lambda\tilde\tau}\,\ud \tilde\tau+I_0[1-G_1(\tau,0)],
\end{align*}
where in the second equality above, we have made a change of variable $\tilde \tau
\to \tau-\tilde \tau$. Since $G$ is a cumulative density function, 
the function $G_1(\tau, \tilde \tau)$ vanishes when $\tau\le \tilde \tau$,
$G_1(\tau, \tilde \tau)$ is
monotonically increasing in the  variable $\tau$ when $\tilde \tau$ is held fixed, 
and is monotonically decreasing in the variable $\tilde \tau$ when $\tau$ is fixed,
and $G_1(\tau, \tilde
\tau )$ tends to $1$ as $\tau$ tends to infinity.\footnote{From a modeling point of view,
one can
regard $G_1(\tau,\tilde\tau )$ as a new family of heterogeneous cumulative distribution functions
coming from some prior probability distributions.
For example, one can take a class of probability density kernels $p=p(\tau,\tau_0)$ which are 
defined on the region $\tau\ge \tau_0$ (think of $\tau_0$ as a parameter marking
the time origin). 
Define $G_1(\tau, \tau_0)= \int_{\tau_0}^{\tau} p(s, \tau_0)ds$. Then clearly such function
$G_1(\tau,\tau_0)$ represents a parameterized  class of cumulative distributions which are heterogeneous in the
sense that it has some dependence on the location $\tau_0$. } To proceed further we 
now dispose of the requirement that $G_1$ was obtained from a cumulative density
function $G$ through a nonlinear map $\varphi$, and keep in stock only some (fairly
mild) consistency
conditions. As will become clear, such relaxation gives rise to a more general family of nonlocal
models which includes the original ones as special cases.  To this end,
denote
\begin{align*}
D_-:=\{ (x, y) \in \mathbb R^2:\, \; 0\le y\le x<\infty\}.
\end{align*}
The following definition seems quite natural.
\begin{defn}[Admissible kernels]
A continuous function $G_1=G_1(\tau,\tilde \tau) $:\, $D_-\to [0,1]$ is said to be admissible if the following conditions hold:
\begin{itemize}
\item $G_1(\tau,\tau)=0$ for any $0\le \tau <\infty$ and $G_1(\tau,\tilde \tau )\to 1$ as 
$\tau \to \infty$ for each fixed $\tilde \tau$.
\item $G_1(\tau, \tilde \tau)$ is non-decreasing in $\tau$ when $\tilde \tau$ is fixed, and
$G_1(\tau, \tilde \tau)$ is non-increasing in $\tilde \tau$ when $\tau$ is fixed.
\end{itemize}
\end{defn}
\begin{rem*}
The continuity requirement on $G_1$ is for convenience and simplicity only. It can certainly
be weakened to partial continuity (in each variables separately) or upper or lower continuity depending
on the convention used in the cumulative distribution function. We shall not dwell on this issue
here.
\end{rem*}

\begin{defn}[$\tau$-model]
Let $\lambda>0$. Let $G_1$ be a given admissible kernel, $S_0>0$, $I_0>0$. We say
($S_1$, $I_1$) evolves according to a $\tau$-model with parameters $(\lambda, G_1)$ 
and initial data ($S_0$, $I_0$) if  $S_1(\tau) = e^{-\lambda \tau} S_0$, and
\begin{align*}
I_1(\tau) 
=\lambda S_0e^{-\lambda\tau}\int_0^\tau[1- G_1(\tau, \tau-\tilde\tau)]e^{ \lambda\tilde\tau}\,\ud \tilde\tau+I_0[1-G_1(\tau,0)], \quad \tau>0.
\end{align*}

\end{defn}
\begin{rem*}
By using the discussion preceding the definition of admissible kernel, it is immediately
clear that the nonlocal model \eqref{hsir} is a special case of our $\tau$-model through
the identification $G_1(\tau,\tilde \tau)= G( \varphi(\tau) -\varphi(\tilde \tau))$. Here one tacitly
assume that the model \eqref{hsir} is already solved on some time interval and there is no difficulty in defining
the nonlinear (albeit not explicit) map $\varphi$. We ignore completely the subtle dependence
of the map $\varphi$ on the cumulative distribution $G$ and treat the resulting $G_1$ function
as given. To obtain more accurate models one can perform iterations on the $\tau$-model
(by feeding iterated trial kernel functions $G_1^{(k)}=G(\varphi^{(k)} (\tau)-\varphi^{(k)} (\tilde \tau))$)
which in some sense will correspond to performing fixed-point type arguments on the original
nonlocal model.
\end{rem*}
\begin{rem*}
Since  $G_1(\tau, \tau-\tilde \tau) \to 0$ as $\tilde \tau \to 0$, we have $I_1(\tau)>0$ for any $0\le \tau <\infty$.  On the other hand, by a change of variable
\begin{align*}
I_1(\tau) = \lambda S_0  \int_0^{\tau} 
[1-G_1(\tau,\tilde \tau)]e^{-\lambda \tilde \tau} d\tilde \tau + I_0 [1-G_1(\tau,0)].
\end{align*}
By using Lebesgue Dominated Convergence, it follows that $I_1(\tau) \to 0$ as $\tau\to \infty$.
It follows that the map
\begin{align*}
\psi (\tau) = \int_0^{\tau} \frac 1 {I_1(\tilde \tau) } d\tilde \tau,
\end{align*} 
provides a monotone and smooth bijection of the time axis $[0,\infty)$.  Denote the inverse map of
$\psi$ as $\psi^{-1}$. Then in the original time scale $t$, we can recover
\begin{align*}
S(t) = S_1(\psi^{-1} (t) ), \quad I(t) = I_1( \psi^{-1} (t) ).
\end{align*}
\end{rem*}

\begin{rem*}
In practice,  the inverse map $\psi^{-1}$ can be recovered more efficiently using a direct
discretization of the $\tau$-variables. More precisely one may (as before in the classic SIR case) first discretize the interval $[0,\infty)$
as $\tau_0, \tau_1,\tau_2, \cdots, $, and solve
\begin{align*}
t_j = \int_0^{\tau_j} \frac 1 {I_1(\tilde \tau) } d\tilde \tau, \quad j=0,1,\cdots.
\end{align*}
Since $ S(t_j) = S_1(\tau_j)$ and $ I (t_j) = I_1(\tau_j)$, the dynamics of $(S, I)$ is then (approximately) fully recovered.
\end{rem*}

Our $\tau$-model shares some general features in common with the classic models. As an example, the following
proposition is a manifestation of the maximum principle in the nonlocal situation. 
\begin{prop}[Maximum principle]
For all $\tau\ge 0$, $I_1(\tau)>0$ and $I_1(\tau) \to 0$ as $\tau \to \infty$. Also for all
$\tau>0$,
\begin{align*}
S_1(\tau)+I_1(\tau) \le S_0+I_0.
\end{align*}
\end{prop}
\begin{proof}
The properties of $I_1$ were proved in one of the remarks before. 
Now since $ G_1\ge 0$, it is clear that
\begin{align*}
S_1(\tau)+I_1(\tau) &\le  \lambda S_0e^{- \lambda\tau }\int_0^\tau e^{\lambda\tilde\tau }\,\ud \tilde\tau+I_0  \notag \\
&= S_0+I_0 - S_0 e^{-\lambda \tau}
\le S_0+I_0\quad\mbox{for all }\tau.
\end{align*}
\end{proof}

\subsection{Simplified $\tau$-Models}
Our $\tau$-model is quite flexible since different choices of the admissible kernel $G_1$ will lead
to
 various new nonlocal SIR models. 
We now choose a homogeneous kernel $G_1(\tau,\tilde \tau)= \tilde G(\tau-\tilde \tau)$,
where $\tilde G:\, [0,\infty) \to [0,1]$ is  a given cumulative distribution function. 
Then clearly
\begin{align} \label{I1_simple}
I_1(\tau)
&=\lambda S_0e^{-\lambda\tau}\int_0^\tau[1-\tilde G(\tilde\tau)]e^{ \lambda\tilde\tau}\,\ud \tilde\tau+I_0[1-\tilde G(\tau)].
\end{align}
The explicit form of $I_1$ can sometimes be worked out by using the specific form of $\tilde{G}$ that one postulates.  Now we discuss a few simple cases which are analogous to the classic SIR
model.

Case 1:  $\tilde G(\tau)=1-e^{-A\tau}$ with $0<A< \lambda $. It is easy to check that
\begin{align*}
I_1=\frac{\lambda S_0}{A-\lambda}e^{-\lambda\tau}+(I_0-\frac{\lambda S_0}{A-\lambda})e^{-A\tau}.
\end{align*}
Clearly  $I_1$ will first increase and then decrease to zero as $\tau\to\infty$. See Figure \ref{fig:3}.
This one corresponds to the case where the
reproduction number is bigger than $1$ in the classic SIR model.
\begin{figure}[H]
  \centering
   \includegraphics[width=6cm]{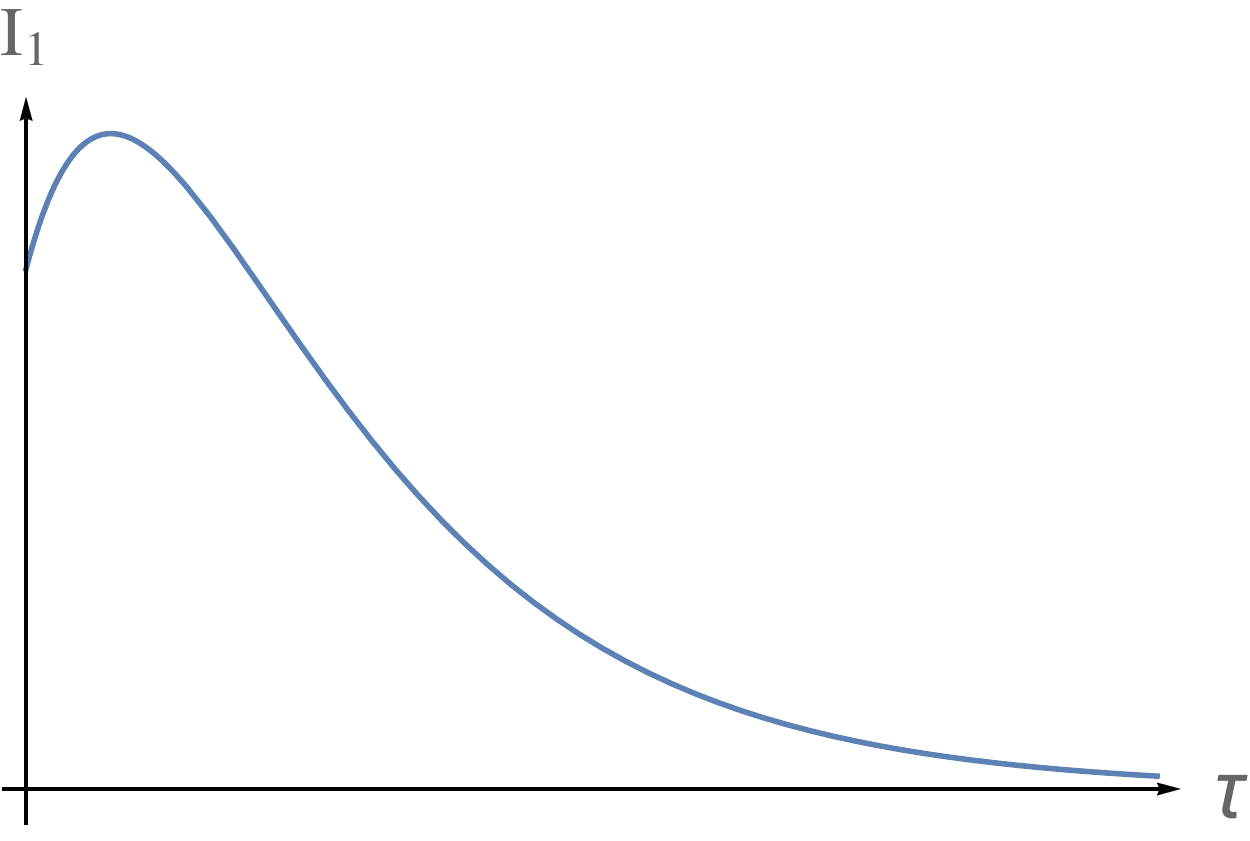}
     \caption{The function $I_1$ in Case 1}\label{fig:3}
\end{figure}

Case 2: $\tilde G(\tau)=1-e^{-A\tau}$ with $\lambda<A<\lambda\left(1+\frac{S_0}{I_0}\right)$ and
\begin{align*}
I_1=\frac{\lambda S_0}{A-\lambda}e^{-\lambda\tau}+(I_0-\frac{\lambda S_0}{A-\lambda})e^{-A\tau}.
\end{align*}
In this case $I_1$ will monotonically decrease to zero as $\tau\to\infty$. See Figure \ref{fig:4}.
This one corresponds to the case where the reproduction
number is less than or equal to $1$ in the classical SIR model. 
\begin{figure}[H]
  \centering
   \includegraphics[width=6cm]{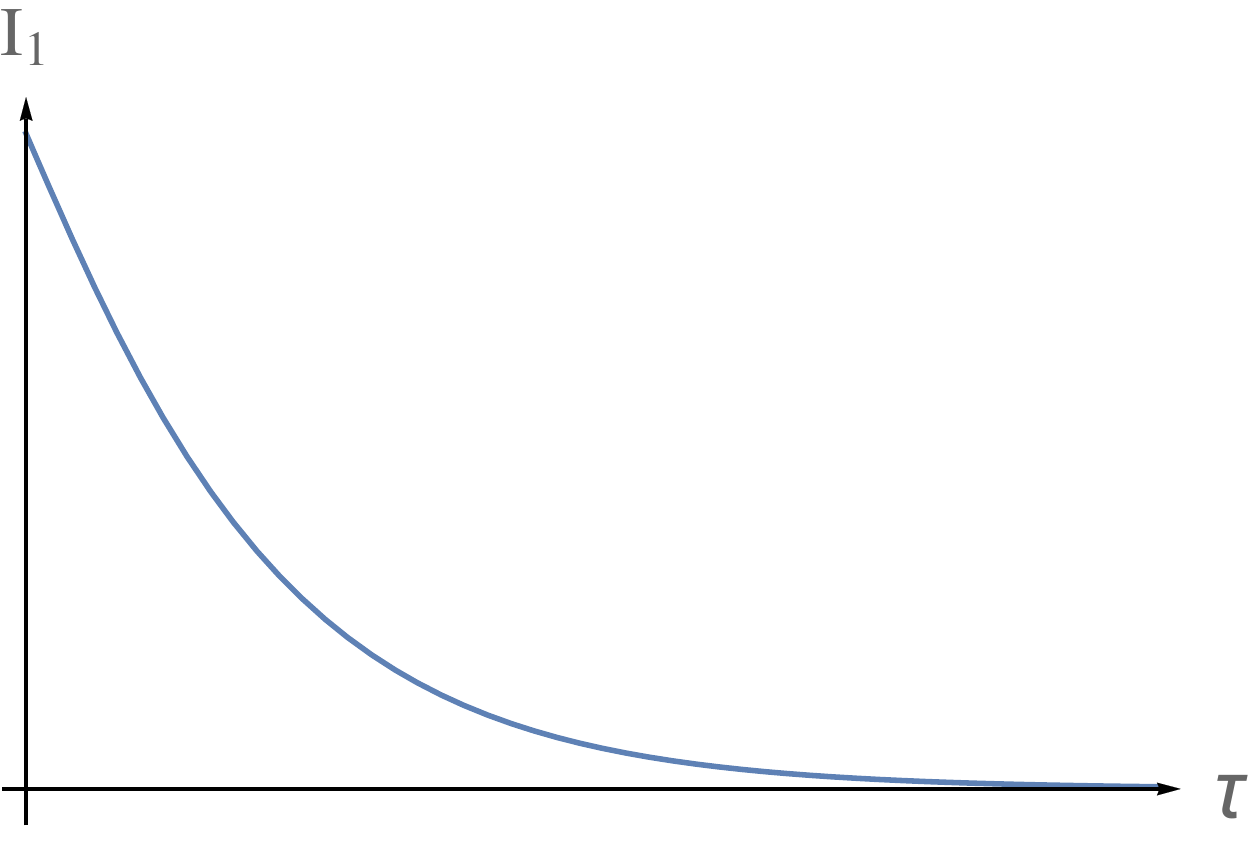}
     \caption{The function $I_1$ in Case 2}\label{fig:4}
\end{figure}

Case 3: An important difference from the classic SIR model is that in our $\tau$-model, $I_1$ may exhibit several bumps, that is, $I_1$ may have several local maximum points. To see this phenomenon in the case of at least two bumps, we  shall need to make some further simplifications
of the model \eqref{I1_simple} by yet another stretching of the time axis. Since this formulation will be needed for the remaining part
of this section, we shall record it here as a proposition. 

\begin{prop}
By using the change of variable $s=s(\tau)=e^{\lambda\tau}$, the model \eqref{I1_simple} is equivalent to the integral equation:
\begin{align} \label{fsgse1}
f(s)= \frac {\beta}s \int_1^s g(\tilde s) \,\ud \tilde s+ g(s), \quad s\ge 1,
\end{align}
where $f(s)= \frac {I_1( \log (s) /\lambda)} {I_0}$, $g(s)=1-\tilde G(\log(s)/\lambda)$  and $\beta=S_0/I_0>0$. Note that $g(s)$ is a given continuous non-increasing function defined on $[1,\infty)$ such that $g(1)=1$ and $g(s)\to 0$ as $s\to\infty$. 
\end{prop}
\begin{proof}
The proof is obvious. One should note that (for later discussions)
 the monotonicity of $I_1(\tau)$ is the same as the one of $f(s)$. Also observe that thanks
 to the transformation $s=e^{\lambda \tau}$, exponential decay in the $\tau$ variable
 then translates to the power decay in the $s$-variable. 
 \end{proof}
 \begin{rem*}
From the relation
\begin{align*}
f(s)=\beta\cdot\frac 1s \int_{1}^s g(\tilde s)\,\ud \tilde s+g(s),
\end{align*}
one sees that $f(s) \gtrsim s^{-1}$ for all $s$, thus $f$ cannot have
better than $s^{-1}$ decay.
\end{rem*}
 \begin{rem*}
 As was already mentioned, our later results of this section will be stated in terms of the pair
 $(g, f)$ with $g$ being a given kernel function and $f$ being the infected population in ``$s$-variable". The main difficulty in our constructions later is the requirement that the kernel function 
 $g(s)$ must monotonically decrease to zero when $s$ tends to infinity. Somewhat surprisingly
 despite this rigid requirement there can still appear oscillations in the infected population by
 judicious choice of the kernel functions. 
  \end{rem*}
  \begin{rem} \label{rem_gf}
 The map $g\to f$ in \eqref{fsgse1} is invertible and in this sense $g$ and $f$ is in one to one
 correspondence.  To see this, one can rewrite \eqref{fsgse1} as
\begin{align*}
(s^{\beta} \int_1^s g(\tilde s) d\tilde s)^{\prime} = f(s) s^{\beta}.
\end{align*}
Integrating and some simple algebra then yield
\begin{align*}
g(s) &= -\beta s^{-\beta-1} \int_1^s f(\tilde s) \tilde s^{\beta} d\tilde s + f(s), \\
g^{\prime}(s) &= (s^{-\beta} \int_1^s \tilde s^{\beta}
f(\tilde s) d\tilde s)^{\prime\prime} \notag \\
&=
\beta(\beta+1) 
s^{-\beta-2} \int_1^s f(\tilde s) \tilde s^{\beta} d\tilde s
- \beta \frac {f(s) } s + f^{\prime}(s).
\end{align*}
\end{rem}

 To see the phenomenon of having at least two peaks in $I_1$, we now give an explicit computable
(albeit not smooth) example in terms of $f$. Choose $\beta=2$ and  a piecewise continuous function $g$ as:
\[
g(s):=
\begin{cases}
\frac 1s,\quad 1\le s\le 2;\\
\frac 12\quad 2\le s\le 2.1;\\
\frac{2.1}{2s}\quad s\ge 2.1.
\end{cases}
\]
Then
\[
f(s):=
\begin{cases}
\frac{2\ln(s)+1}{s},\quad 1\le s\le 2;\\
\frac 32+\frac{2\ln 2-2}{s}\quad 2\le s\le 2.1;\\
\frac 2s \left(\ln 2+\frac{1}{20}+\frac{2.1}{2}\ln(\frac{s}{2.1})\right)+\frac{2.1}{2s}\quad s\ge 2.1.
\end{cases}
\]
The picture of this $f(s)$ and $I_1(\tau)$ can be illustrated in Figure \ref{fig:5} below:
\begin{figure}[H]
  \centering
   \includegraphics[width=6cm]{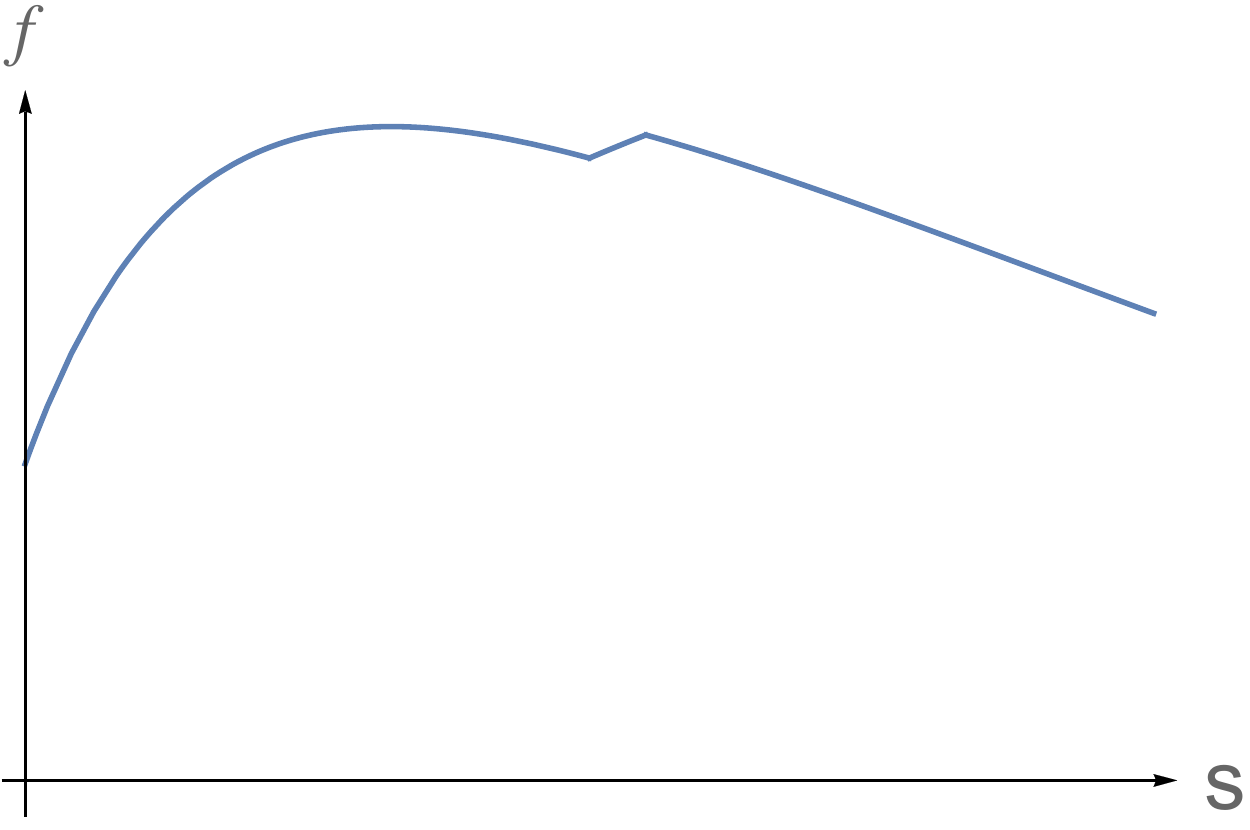}\hfill \includegraphics[width=6cm]{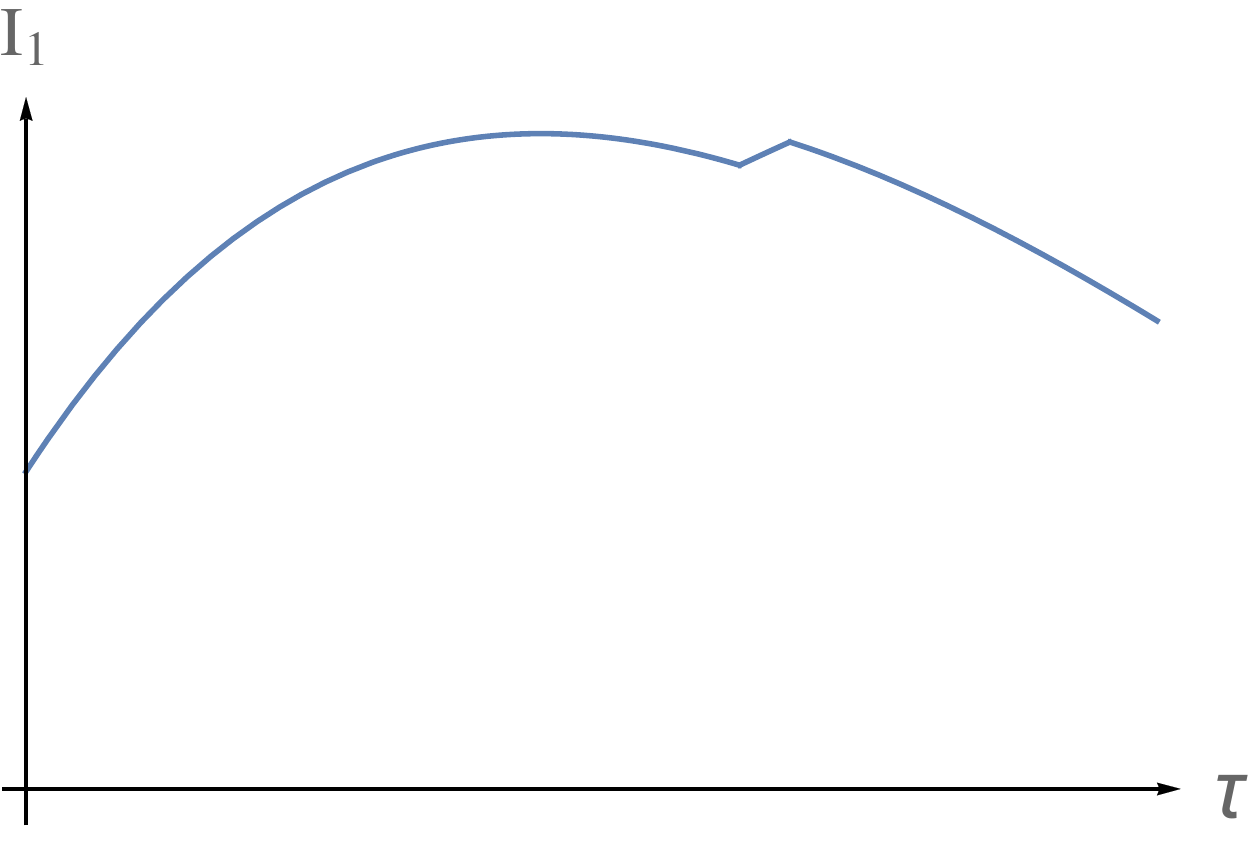}
     \caption{The functions $f$, and $I_1$ (with $\lambda=0.1$), in Case 3}\label{fig:5}
\end{figure}

Although the above $g$ and $f$ are not smooth, it provides some good evidence that one
can hope to construct smooth solutions with two or more peaks. One should note, however, that
a direct mollification of the $g$ function may introduce $O(1)$-changes for the derivative of $f$ near its
original peak and thus may destroy the peak. On the other hand, if one prefers to mollifying directly the function
$f$ near its two peaks, then one has to make sure that $g$ still remains monotonically decreasing
after mollification, that is, 
\begin{align*}
g^{\prime}(s) =\beta(\beta+1) 
s^{-\beta-2} \int_1^s f(\tilde s) \tilde s^{\beta} d\tilde s
- \beta \frac {f(s) } s + f^{\prime}(s) \le 0, \quad \forall\, s.
\end{align*}
Now note that near the second peak of $f$, $f^{\prime}$ is not continuous and  it has
$O(1)$-jump around the peak. Since $g^{\prime}(2.1-)=0$, the above inequality may
well stop being true after mollification due to this $O(1)$-change.
For these considerations we shall present in detail
the construction and also the mollification procedure in the remaining part of this section.
Our main results in the next subsection are Proposition \ref{lem51} and Proposition 
\ref{prop53} which give the existence of solutions having multiple peaks. To clarify the meaning
of peak solution, we now introduce a simple definition.

\begin{defn}[Peak]
A point $s_0 \in (1,\infty)$ is said to be a peak (i.e. strict local max) for the function $f:\, [1,\infty) \to (0,\infty)$,
if there exists a local neighborhood $\mathcal N_0$ of $s_0$, such that $f(s)<f(s_0)$
for any $s \in \mathcal N_0$, $s\ne s_0$. In particular, if $f$ is $C^2$, $f^{\prime}(s_0)=0$ and
$f^{\prime\prime}(s_0)<0$, then $s_0$ is a peak for $f$. 
\end{defn}
Our main objective is to find  kernel functions $g$ which decrease monotonically to zero
such that the resulting infected population $f$ possesses multiple peaks.  Thanks to the algebra
in Remark \ref{rem_gf},  it suffices for us to construct a positive  function $f$ 
with multiple peaks and obeying the nonlocal
inequality $g^{\prime} <0$ (along with decay conditions at infinity) i.e.:
\begin{align*}
\beta(\beta+1) 
s^{-\beta-2} \int_1^s f(\tilde s) \tilde s^{\beta} d\tilde s
- \beta \frac {f(s) } s + f^{\prime}(s) < 0, \quad \forall\, s \ge 1.
\end{align*}
Achieving the above inequality point-wisely together with $C^{\infty}$-smoothness turns out
to be a delicate task, and to elucidate the main ideas we shall present a rough result first
(Proposition \ref{lem51}) which will be upgraded later in Proposition \ref{prop53}. 
In both propositions we shall slightly change the notation and denote $f$ by $I$, that is, 
\begin{align*}
I(s)= \frac {\beta} s \int_1^s g(\tilde s) d\tilde s+g (s), \quad s\ge 1;
\end{align*}
and equivalently 
\begin{align*}
g(s) &= -\beta s^{-\beta-1} \int_1^s I(\tilde s) \tilde s^{\beta} d\tilde s + I(s), \\
g^{\prime}(s) &= (s^{-\beta} \int_1^s \tilde s^{\beta}
I(\tilde s) d\tilde s)^{\prime\prime} \notag \\
&=
\beta(\beta+1) 
s^{-\beta-2} \int_1^s I(\tilde s) \tilde s^{\beta} d\tilde s
- \beta \frac {I(s) } s + I^{\prime}(s).
\end{align*}
\subsection{Peak Solutions}

Case 4a: existence of arbitrarily (finitely) many peaks. We first show that it is possible to admit
finitely many peaks (as many as possible).  A stronger result will be established in 
forthcoming Case 4b (compare with Proposition \ref{prop53}). 

\begin{prop}[Finitely many peaks, rough control] \label{lem51}
Let $\beta> 0$. For any integer $m_0\ge 1$ and any $0<\theta\le 1$, there exist a family of 
$C^{\infty}$-smooth functions
$I=I(s): \, [1,\infty) \to (0,\infty)$ such that
\begin{align} \label{lem51eq0}
&  \left(s^{-\beta} \int_1^s I(\tilde s ) \tilde s^{\beta} d\tilde s\right)^{\prime\prime} \notag \\
=&\;
\beta(\beta+1) s^{-\beta-2}\int_1^s I(\tilde s) \tilde s^{\beta} d\tilde s
-\beta s^{-1} I(s) + I^{\prime}(s)  <0,
\quad\forall\, 1\le s<\infty.
\end{align}
Furthermore $I(s) \lesssim s^{-\theta}$, $|I^{\prime}(s)|
\lesssim s^{-1-\theta}$, and $I(s)$ has at least $m_0$ strict local maxima located
at some points $s_1$, $s_2$, $\cdots$, $s_{m_0}$, such that
$I^{\prime}(s_j)=0$, $I^{\prime\prime}(s_j)<0$ for any $1\le j \le m_0$. 
\end{prop}
\begin{rem*}
First observe that it suffices for us to construct
a piecewise smooth $I(s)$ (overall as a Lipschitz continuous function) which satisfies \eqref{lem51eq0} except at one point $s=\overline S_0$. Then mollifying such $I(s)$ around $s=\overline S_0$ suitably easily yields
the desired smooth function still satisfying \eqref{lem51eq0}. Indeed, suppose  $I(s)$  satisfies \eqref{lem51eq0} for $1\le s<\overline S_0$ and $s>\overline S_0$ respectively,  $I(s)$ decays to zero as
$s$ tends to infinity and $I(s)$ has $m_0$ local maxima inside the interval $(1,\overline S_0)$. Also assume $I(\cdot)$ is $C^{\infty}$
for $s<\overline S_0$ and $s>\overline S_0$ respectively, and the LHS of \eqref{lem51eq0} is bounded
above by some uniform negative constant for $s$ close to $\overline S_0$.  Then 
\begin{align*}
g(s) &= \left(s^{-\beta} \int_1^s I(\tilde s) \tilde s^{\beta} d\tilde s\right)^{\prime}
= -\beta s^{-\beta-1} \int_1^s I (\tilde s ) \tilde s^{\beta} d\tilde s
+I(s)
\end{align*}
satisfies $g^{\prime}(s)<0$ for $s<\overline S_0$ and $s>\overline S_0$, and $g^{\prime}(s)<-c_0<0$
for $s$ close to $\overline S_0$ ($c_0>0$ is some constant). This (together with the decay of $g$ at
infinity) implies that $g$ is Lipschitz, $g>0$ for all $s$ and $g$ is strictly monotonically decreasing. 
Define $g_{\operatorname{smooth} }$ to be a suitable mollification of $g$ (the mollification is only
done in a sufficiently small neighborhood of $s=\overline S_0$). Define 
\begin{align*}
I_{\operatorname{smooth} }(s) = s^{-\beta} \cdot \left(s^{\beta} \int_1^s g_{\operatorname{smooth} }(\tilde s) d\tilde s\right)^{\prime}
&= \beta \frac {\int_1^s g_{\operatorname{smooth} }(\tilde s) d\tilde s } {s} +g_{\operatorname{smooth} }(s).
\end{align*}
Then $I_{\operatorname{smooth} }$ is clearly positive, $C^{\infty}$ and satisfies \eqref{lem51eq0}. Furthermore
 $I_{\operatorname{smooth} }$ has at least $m_0$ local maxima since it inherits the local maxima of $I(\cdot)$ inside the interval $(1,\overline S_0)$ as long as the mollification is done away from the local maxima (note that $I_{\operatorname{smooth} }$ may generate some
 new maxima near the point $\overline S_0$ since mollification is done for the $g$ function). 
\end{rem*}

\begin{rem*}
We clarify a bit more in detail the ``suitable" mollification procedure for the $g$-function alluded to
in the previous remark. The main point is to preserve strict monotonicity whilst achieving $C^{\infty}$-smoothness. To be precise, assume $g$ is $C^{\infty}$ except at $s=\overline S_0$, $g^{\prime}<0$
for $s\ne\overline S_0$, and $g^{\prime}(s) <-c_0<0$ for $|s-\overline S_0|<2\delta_0$ ($\delta_0>0$
is sufficiently small). Choose a $C^{\infty}$-smooth radial decreasing bump function $\phi_0$ such that $\phi_0 (x)= \exp(- \frac  1 {1-x^2 } )$ for $\frac 23 <|x|<1$, $\phi_0(x)=1$ for $|x|<\frac 13$,
 and $\phi_0(x)=0$
for $|x|\ge 1$. Define
\begin{align*}
h(s) = g^{\prime}(s) \left(1- \phi_0( \frac {s-\overline S_0} {\delta_0} ) \right) -a\cdot \phi_0(\frac {s-\overline S_0}
{{\delta_0} } ),
\end{align*}
where the constant $a>0$ is chosen such that
\begin{align*}
\int_{\overline S_0-\delta_0}^{\overline S_0+\delta_0} h(s) ds = \int_{\overline S_0-\delta_0}^{\overline S_0+\delta_0}
g^{\prime}(s) ds.
\end{align*}
Define $c_1= \min\{ c_0, \, a\} >0$. Then clearly for all $|s-\overline S_0|\le \delta_0$, one has
\begin{align*}
h(s) < -c_1<0.
\end{align*}
Also $h(s) =g^{\prime}(s)$ for $|s-\overline S_0|>\delta_0$.  Define 
\begin{align*}
g_{\operatorname{smooth} }(s) = g(1) + \int_1^s h(\tilde s) d\tilde s.
\end{align*}
Then it is not difficult to check that $g_{\operatorname{smooth} }(s) = g(s) $ for $s<\overline S_0-\delta_0$ or
$s>\overline S_0+\delta_0$, and $g_{\operatorname{smooth} }^{\prime}<0$ for all $s$. Thus $g_{\operatorname{smooth} }$ is the desired $C^{\infty}$ function.
\end{rem*}
\begin{rem*}
The basic idea of the construction of the function $I(s)$ is as follows.  First if we ignore the constraint of having
$m_0$ local maxima, we can choose $I(s)$ piece-wisely as
\begin{align*}
I(s)=
\begin{cases}
1, \quad 1\le s<\overline S_0, \\
(\frac {\overline S_0} s )^{\theta}, \quad s>\overline S_0,
\end{cases}
\end{align*}
where $0<\theta\le 1$.  It is easy to check that the inequality \eqref{lem51eq0}  holds for $s<\overline S_0$. For
$s>\overline S_0$, one observes that
\begin{align*}
s^{-\beta} \int_1^s I(\tilde s) \tilde s^{\beta} d\tilde s 
&= s^{-\beta} \int_1^{\overline S_0} \tilde s^{\beta} d\tilde s + s^{-\beta}
\int_{\overline S_0}^s \tilde s^{\beta} \cdot (\frac {\overline S_0} {\tilde s} )^{\theta} d \tilde s \notag \\
&= s^{-\beta}\left( \frac{\overline S_0^{\beta+1}-1} {\beta+1}-
\frac {\overline S_0^{\beta+1}}{\beta-\theta+1} \right) + s^{1-\theta} \frac {\overline S_0^{\theta}} {\beta-\theta+1}.
\end{align*}
This then clearly yields \eqref{lem51eq0} thanks to the condition $0<\theta\le 1$.  Now
to finish the construction we can use an ``$\epsilon$-perturbation" idea to create $m_0$
local maxima in the range $1\le s<\overline S_0$. The mollification then shows that one has
at least $m_0$ local maxima. 

\end{rem*}
\begin{proof}
Define 
\begin{align*}
I(s)=
\begin{cases}
1+\epsilon \cdot e^{-s} \sin s, \quad 1\le s\le \overline S_0; \\
\alpha_1 (\frac {\overline S_0}s)^{\theta}, \quad s>\overline S_0,
\end{cases}
\end{align*}
where $\alpha_1= 1+\epsilon e^{-\overline S_0} \sin \overline S_0$,
and the parameters $(\epsilon,\overline S_0)$ will be specified momentarily. It will
become clear from the description below that  we obtain a family of functions. Furthermore 
 $I(s)$ is piecewise $C^{\infty}$  and by using mollification (see the Remarks before)
 $I(s)$ can be made arbitrarily smooth.

First it is easy to check that the function $e^{-s} \sin s$ has strict local maxima at the points
$s= (n+\frac 14) \pi$, $n=0,2,4,\cdots$. We can take $\overline S_0= (2(m_0-1)+\frac 14+\eta_0)\pi$
with $\eta_0>0$ being  sufficiently small.  This obviously guarantees that $I(\cdot)$ admits precisely
$m_0$ strict local maxima located inside the interval $(1,\overline S_0)$. 

Next observe that in \eqref{lem51eq0} if we plug in the main order $I(s)\equiv 1$ we
obtain the value $-\beta$. Thus by choosing $0<\epsilon<1$ sufficiently small ($\epsilon$
can depend on $\overline S_0$ and $\beta$) one can guarantee that the LHS of \eqref{lem51eq0} is always
less than $-\beta/2$ for $1\le s< \overline S_0$.

Finally we need to check \eqref{lem51eq0} for $\overline S_0<s<\infty$.  Clearly
\begin{align*}
s^{-\beta} \int_1^s I(\tilde s) \tilde s^{\beta} d\tilde s 
&= s^{-\beta} \int_1^{\overline S_0} \tilde s^{\beta} d\tilde s + 
s^{-\beta} \int_1^{\overline S_0} \tilde s^{\beta}
\cdot \epsilon e^{-\tilde s} \sin \tilde s d \tilde s+
\alpha_1 s^{-\beta}
\int_{\overline S_0}^s \tilde s^{\beta} \cdot (\frac {\overline S_0} {\tilde s} )^{\theta} d \tilde s \notag \\
&= s^{-\beta}\left( \frac{\overline S_0^{\beta+1}-1} {\beta+1}-\alpha_1
\frac {\overline S_0^{\beta+1}}{\beta-\theta+1} 
+C_1\cdot \epsilon\right) + \alpha_1 s^{1-\theta} \frac {\overline S_0^{\theta}} {\beta-\theta+1},
\end{align*}
where $C_1= \int_1^{\overline S_0} \tilde s^{\beta} e^{-\tilde s} \sin \tilde s d\tilde s$. 
By choosing $\epsilon<\frac 1 {2(\beta+1)} \cdot \frac 1 {|C_1|+1}$, one can clearly
fulfill \eqref{lem51eq0} for $\overline S_0<s<\infty$. Furthermore the LHS of \eqref{lem51eq0}
is bounded above by some uniform negative constant for $s$ close to $\overline S_0$.

\end{proof}

Case 4b: More refined results: existence of arbitrarily finitely many peaks with precise control
of the location of maxima. 

\begin{lem} \label{lem52}
Let $\beta>0$ and $1<\overline S_0<\infty$.  There exists $\theta_0=\theta_0(\beta,\overline S_0) 
\in (0,1)$, such that for any $0<\theta<\theta_0$, there exists 
$f:\, [1,\infty)\to (0,\infty)$ satisfying the following properties:

\begin{itemize}
\item $f\in C^{\infty}$, 
 and $|f(s)| \lesssim s^{-\theta}$, $|f^{\prime}(s)| \lesssim s^{-\theta-1}$ for all $s$. 

\item For some $0<\eta_0<1$ ($\eta_0>0$ is a tunable parameter which can be
made as small as possible), one has $f(s)\equiv 1$ for $1\le s \le \overline S_0-\eta_0$,
and $f^{\prime}(s)<0$ for all $s>\overline S_0-\eta_0$. 

\item  There exists a constant $c_1>0$, such that for all $1\le s <\infty$,
\begin{align} \label{lem52e1}
\left(s^{-\beta} \int_1^s f(\tilde s )\tilde s^{\beta} d\tilde s\right)^{\prime\prime}
+ c_1 s^{-\theta-1} <0.
\end{align}

\end{itemize}

\end{lem}
\begin{proof}
We first define a $C^1$-function as:
\begin{align*}
f_0(\tau)= \begin{cases}
1, \quad 1\le \tau \le \overline S_0,\\
\frac 1 {1-\theta} ( ( \frac {\overline S_0} {\tau} )^{\theta} - \theta \frac {\overline S_0} {\tau} ),
\quad \tau > \overline S_0,
\end{cases}
\end{align*}
where $\theta \in (0,1)$ will be taken sufficiently small.  For \eqref{lem52e1}
the regime $s\le \overline S_0$ is
easy to check. For $s>\overline S_0$, we have
\begin{align*}
&s^{-\beta} \int_1^s f_0(\tilde s )\tilde s^{\beta} d\tilde s \notag \\
=&\; s^{-\beta} \left( \frac {\overline S_0^{\beta+1}-1} {\beta+1}
- \frac 1 {1-\theta} \cdot \frac 1 {\beta-\theta+1} \overline S_0^{\beta+1} 
+ \frac {\theta}{1-\theta} \cdot \frac 1 {\beta} \overline S_0^{\beta+1} \right)   \notag \\
& \quad + \frac 1 {1-\theta} \cdot \frac 1 {\beta-\theta+1} \cdot \overline S_0^{\theta}
\cdot s^{1-\theta}  - \frac {\theta}{1-\theta} \cdot \overline S_0 \cdot \frac 1 {\beta}.
\end{align*}
Now note that if we take $\theta$ such that
\begin{align*}
\frac {\theta}{1-\theta} \cdot \frac 1 {\beta} \overline S_0^{\beta+1} \le \frac 1 {\beta+1},
\end{align*}
then we obtain
\begin{align*}
&s^{-\beta} \int_1^s f_0(\tilde s )\tilde s^{\beta} d\tilde s \notag \\
=&\;  -A_1 s^{-\beta} +A_2 s^{1-\theta} - A_3,
\end{align*}
where $A_i>0$ are constants. This clearly yields a concave function
satisfying \eqref{lem52e1}.

Now to finish the construction we just need to mollify $f_0$ around $\overline S_0$ whilst keeping
the constraint \eqref{lem52e1} (with a smaller constant $c_1$ if necessary). Let
$\phi_0\in C_c^{\infty}(\mathbb R)$ be radial decreasing such that $\phi_0 (x)= \exp(- \frac  1 {1-x^2 } )$ for $\frac 23 <|x|<1$, $\phi_0(x)=1$ for $|x|<\frac 13$,
 and $\phi_0(x)=0$
for $|x|\ge 1$.

Take small constant $\eta_0>0$ and define
\begin{align*}
h(s) = f^{\prime}_0(s) \left(1-\phi_0(\frac {s-\overline S_0} {\eta_0} ) \right) -a \phi_0(\frac{s-\overline S_0}{\eta_0}),
\end{align*}
where $a>0$ satisfies
\begin{align*}
a \int \phi_0(x) dx =- \int f_0^{\prime}(\overline S_0+\eta_0 x)  \phi_0(x) dx.
\end{align*}
Observe that $h(s)<0$ for all $s>\overline S_0-\eta_0$. 
Define $f(s)= f_0(1) + \int_1^s h(\tilde s) d\tilde s$. Clearly $f(s)=f_0(s)$,
$f^{\prime}(s)=f_0^{\prime}(s)$ for $|s-\overline S_0|>\eta_0$. For $|s-\overline S_0|\le \eta_0$,
we have $|f^{\prime}(s)-f_0^{\prime}(s)| =O(\eta_0)$, $|f(s)-f_0(s)|=O(\eta_0^2)$. 
Then for $|s-\overline S_0|<\eta_0$, we have
\begin{align*}
&  \left(s^{-\beta} \int_1^s f(\tilde s ) \tilde s^{\beta} d\tilde s\right)^{\prime\prime} \notag \\
=&\;
\beta(\beta+1) s^{-\beta-2}\int_1^s f(\tilde s) \tilde s^{\beta} d\tilde s
-\beta s^{-1} f(s) + f^{\prime}(s) \notag \\
=& \;
\beta(\beta+1) s^{-\beta-2}\int_1^s f_0(\tilde s) \tilde s^{\beta} d\tilde s
-\beta s^{-1} f_0(s) + f^{\prime}_0(s) +O(\eta_0).
\end{align*}
By further shrinking $\eta_0$ it is then easy to fulfill the constraint \eqref{lem52e1}.
\end{proof}
\begin{rem*}
Alternatively, one can use the usual convolution to define (a slightly different version of ) $f(s)$ as follows:
first extend $f_0(s)$ to $\mathbb R $ such that $f_0(s)=1$ for $s\le 0$. Then for $\eta_0>0$
 let $\phi_{\eta_0}(x)=\eta_0^{-1} \phi_0(x/\eta_0)/\|\phi_0\|_1$. 
Define $f$ as the restriction of the convolution $f_0*\phi_{\eta_0}$ to the axis $[1,\infty)$. 
Note that $f\equiv 1$ on $[1,\overline S_0-\eta_0]$ and $f^{\prime}<0$ on $(\overline S_0-\eta_0,\infty)$. The
estimates
\begin{align*}
 & |f(s)-f_0(s)|\lesssim s^{-1-\theta} \cdot \eta_0, \quad\forall\, s, \\
 & |f^{\prime}(s) -f_0^{\prime}(s)| \lesssim s^{-2-\theta} \cdot \eta_0, \quad
 \forall\, s >\overline S_0+\eta_0, \\
 & |f^{\prime}(s)| \lesssim \eta_0, \quad \forall\, \overline S_0-\eta_0\le s \le \overline S_0+\eta_0,
\end{align*}
then yield the result.
\end{rem*}

The following is a strengthened version of Proposition \ref{lem51}.

\begin{prop}[Finitely many peaks,  precise control] \label{prop53}
Let $\beta> 0$. For any $1<M_1<M_2<\infty$, any integer $m_0\ge 1$, 
and any $m_0$ different points $\{s_j\}_{j=1}^{m_0}$ (WLOG $s_1<s_2<\cdots
<s_{m_0}$) in the interval
$(M_1,M_2)$, one can find a $C^{\infty}$-smooth function
$I=I(s): \, [1,\infty) \to (0,\infty)$ such that the following hold:

\begin{itemize}
\item For all $1\le s<\infty$, 
\begin{align} 
&  \left(s^{-\beta} \int_1^s I(\tilde s ) \tilde s^{\beta} d\tilde s\right)^{\prime\prime} \notag \\
=&\;
\beta(\beta+1) s^{-\beta-2}\int_1^s I(\tilde s) \tilde s^{\beta} d\tilde s
-\beta s^{-1} I(s) + I^{\prime}(s)  <0. \label{prop53e0}
\end{align}
\item  For some $0<\theta<1$, $I(s) \lesssim s^{-\theta}$, $|I^{\prime}(s)|
\lesssim s^{-1-\theta}$, for all $s\ge 1$.
\item  $I(s)$ has  $m_0$ strict local maxima located
at the points $s_1$, $s_2$, $\cdots$, $s_{m_0}$, such that
$I^{\prime}(s_j)=0$, $I^{\prime\prime}(s_j)<0$ for any $1\le j \le m_0$. 
\item $I(s)$ has no other peaks besides the points $\{s_j\}_{j=1}^{m_0}$. More precisely,
there are points $a_1<b_1<a_2<b_2<\cdots<a_{m_0}<b_{m_0}<a_{m_0+1}$, such that
\begin{itemize}
\item $I(s)=1$ for $1\le s\le a_1$, $b_j\le s\le a_{j+1}$ ($j=1,\cdots,m_0$);
\item $I(a_{m_0+1})=1$ and $I^{\prime}(s)<0$ on $(a_{m_0+1}, \infty)$;
\item On each interval $(a_j, b_j)$, $I(s)$ has a peak at $s=s_j$, and $I^{\prime}(s)
>0$ for $a_j<s<s_j$, $I^{\prime}(s)<0$ for $s_j<s<b_j$. 
\end{itemize}
\end{itemize}
\end{prop}
\begin{proof}
We first let $\overline S_0=M_2$ and choose $f(s)$ by using Lemma \ref{lem52}. Take a
radial decreasing bump function\footnote{Note that here we choose the
bump function $\phi_1$ to have a ``strict" peak instead of a flat maximum at
the center.}
$\phi_1\in C_c^{\infty}(\mathbb R)$ with $\phi_1(x)= e^{-\frac 1 {1-x^2}}$ for $|x|<1$,
$\phi_1(x)=0$ for $|x|>1$. 
Define
\begin{align*}
I(s) = f(s) + \sum_{j=1}^{m_0} \delta_0^3 \phi_1( \frac {s-s_j} {\delta_0} ),
\end{align*}
where $\delta_0 \ll \min_{0\le j\le m_0} |s_{j+1}-s_j|$ (here $s_0=M_1$, 
$s_{m_0+1}=M_2-\eta_0$ where $\eta_0$ is in Lemma \ref{lem52}) 
 will be taken sufficiently small (such that 
  \eqref{prop53e0} is  satisfied).  For $1\le j\le m_0$, define
 $a_j= s_j -\delta_0$, $b_j=s_j+\delta_0$. Define $a_{m_0+1}=M_2-\eta_0$. 
It is then not difficult to check that $I(s)$ fulfills the desired properties.
\end{proof}
\begin{rem} \label{rem_infty_1}
It is even possible to modify slightly the above construction and
 construct a function $I(s)$ having infinitely many peaks within $(M_1, M_2)$. Take
 a sequence of increasing points $s_j \to s_{\infty} \in (M_1,M_2)$.  Define 
 \begin{align*}
 I(s) = f(s) + \sum_{j=1}^{\infty} c_j \phi_1(\frac {s-s_j} {\delta_j} ),
 \end{align*}
where $0<\delta_j \ll \min\{s_{j+1}-s_j, \, s_j-s_{j-1},\, 2^{-j} \}$ for all $j$ (so that bubbles stay
disjoint) and $c_j=e^{-2/{\delta_j}}$ (one may also need to shrink
$\delta_j$ further so that \eqref{prop53e0} is satisfied). Easy to check that for any integer $k\ge 0$,
\begin{align*}
\sum_{j} c_j  \delta_j^{-k} \lesssim \sum_{j} k! \cdot e^{-\frac 1 {\delta j} } <\infty.
\end{align*}
Thus $I$ is $C^{\infty}$-smooth which has infinitely (small) peaks located at the points
$s_j$.
\end{rem}

\begin{rem*}
Another natural idea is to look for a function $I(s)$ possessing infinitely many plateau segments
emanating to infinity  (from which one can create infinitely many peaks in each plateau
such that  the formed peaks will have  positive distances from
each other). However one can \emph{disprove} such a possibility. More precisely, recall
that 
\begin{align*}
I(s) = \beta \cdot \frac 1 {s} \int_1^s g(\tilde s) d\tilde s + g(s), \quad\forall\, s\ge 1,
\end{align*}
and $g>0$ monotonically decrease to zero as $s$ tends to infinity. 

Claim: The function $I(s) $ cannot have infinitely many intervals $[a_n,b_n]$ with $a_n, \,b_n \to \infty$ such
that $I (s) \equiv c_n$ for some constant $c_n$ on the interval $[a_n, b_n]$. 

\begin{proof}[Proof of Claim.]
Argue by contradiction. Suppose such $I(s)$ exists. Then on each $[a_n,b_n]$, we have
\begin{align*}
( s^{\beta} \int_1^s g(\tilde s) d\tilde s)^{\prime} = c_n s^{\beta}.
\end{align*}
Solving this on $[a_n,b_n]$ then gives
\begin{align*}
&\int_1^s g(\tilde s) d\tilde s = \frac {c_n}{\beta+1} s -d_n s^{-\beta}, \quad s\in [a_n,b_n];\\
&g(s) = \frac {c_n} {\beta+1} + \beta d_n s^{-\beta-1}, \quad s\in [a_n,b_n],
\end{align*}
where $d_n\ge 0$ (since $g$ is non-increasing). It follows that
\begin{align*}
\frac 1s \int_1^{s} g(\tilde s) d\tilde s \le g(s), \quad \text{for $s=a_n$, $n\ge 1$}.
\end{align*}
This is impossible since there exists some number $\alpha_0>0$, such that
\begin{align} \label{egs_alpha0}
\frac 1 s \int_1^s g(\tilde s) d\tilde s >g(s), \quad \forall\, s>\alpha_0.
\end{align}
This last inequality can be easily proved by an $\epsilon$-$\delta$ argument using the monotonicity of
$g$ and the fact that $g$ tends to zero as $s \to \infty$. 
\end{proof}
\end{rem*}
\begin{rem*}
Inspired by the preceding remark, one can prove a stronger result which asserts that
the function $I(s)$ must be strictly monotonically decreasing near infinity. Thus the construction in Remark \ref{rem_infty_1} is in some sense the best possible since one cannot have peaks
propagating to infinity.

Claim: There exists $\alpha_0>0$, such that $I^{\prime}(s)<0$ for all $s>\alpha_0$.

\begin{proof}[Proof of Claim]
This follows directly from the identity 
\begin{align*}
I^{\prime}(s)= g^{\prime}(s)+ \frac {\beta} s \left(  -\frac 1 {s} \int_1^s g(\tilde s)
d\tilde s + g(s) \right).
\end{align*}
and the inequality \eqref{egs_alpha0}.
\end{proof}

\end{rem*}

\begin{rem*}
From the point of view of epidemic modeling,  one should perhaps not expect infinitely many
(relatively) large outbreaks which is naturally excluded in our nonlocal model.
\end{rem*}

The above concludes our discussion on the properties and relations of $I(s)$ and the kernel
$g(s)$. From a practical point of view, one can collect data for (suitably normalized infected 
population) $I(s)$, and reconstruct the approximate kernel $g(s)$. In this way one can build
a nonlocal model to make future predictions of possible outbreaks.

\subsection*{Concluding Remarks}
In this work we analyzed the classic SIR model and its nonlocal variant recently introduced 
in \cite{Guan}. The latter is a hybrid differential-integral model incorporating general probability
transition kernels as parameters. We discussed two approaches. One is based on quantifying the explicit relation of the infected population versus the susceptible population. The basic observation
is that the susceptible population strictly decreases in time which makes it a natural time arrow. 
In the classic SIR model it is possible to give almost explicit and analytic expressions
of the infected population in terms of the susceptible population. In the nonlocal SIR model one
can use an approximation of the kernel to derive the explicit functional relationship between
the infected population and the susceptible population.   In our second approach, we introduced
a nonlocal time parameter whose differential is the reciprocal of the infected population. 
Under this nonlocal stretching of the time axis, the classic SIR model becomes completely linear
and the dynamics can be easily classified. The functional relation between the new time arrow and
the old time can be explicitly worked out by using the classification result obtained in the first
step. Exploiting a similar idea we analyzed the nonlocal
SIR model and discovered several novel peak solutions.  Compared with the classic SIR model an interesting new feature of the nonlocal SIR model is the appearance of solutions having multiple
peaks. The stability/instability of such solutions as well as the classification of general asymptotics of
such solutions are presently not known and these seem to be an interesting new direction. Another fundamental issue left unaddressed
in this work is the notion of the reproduction number in general nonlocal epidemic 
models. Identifying the correct generalization of the reproduction number as well as building the natural connection between the nonlocal and the classic epidemic models seem to be a challenging task. We hope to address some of these issues in future works.

\bibliographystyle{abbrv}

\bigskip

\noindent L. Guan

\noindent Mathematics Department, Tulane University  \\
New Orleans, LA 70118, USA\\[1mm]
Email:  \textsf{lguan@tulane.edu}

\bigskip

\noindent D. Li

\noindent Department of Mathematics, The Hong Kong University of Science and Technology\\
Clear Water Bay, Kowloon, Hong Kong\\[1mm]
Email: \textsf{madli@ust.hk}

\bigskip

\noindent K. Wang

\noindent Department of Mathematics, The Hong Kong University of Science and Technology\\
Clear Water Bay, Kowloon, Hong Kong\\[1mm]
Email: \textsf{kewang@ust.hk}

\bigskip

\noindent K. Zhao

\noindent Mathematics Department, Tulane University\\
New Orleans, LA 70118, USA\\[1mm]
Email:  \textsf{kzhao@tulane.edu}

\end{document}